\documentclass[12pt]{amsart}
\usepackage[top=1in,bottom=1in, left=1 in, right= 1 in, includehead,includefoot]{geometry}

\usepackage{amsthm}
\usepackage{amsmath,amssymb}
\usepackage[T1]{fontenc}
\usepackage{dsfont}
\usepackage{color,hyperref}
\usepackage[numbers,sort&compress]{natbib}
\usepackage{enumerate}

\newtheorem{theorem}{Theorem}[section]

\newtheorem{lemma}[theorem]{Lemma}

\theoremstyle{definition}

\theoremstyle{remark}
\newtheorem{remark}[theorem]{Remark}
\theoremstyle{remark}

\numberwithin{equation}{section}

\newcommand{\ind}{{\bf 1}}

\def\inddd#1{{\ind}_{\left\{#1\right\}}} 
\newcommand{\proba}{\mathbb P}
\newcommand{\esp}{{\mathbb E}}

\newcommand{\inv}{^{-1}}
\newcommand{\cov}{{\rm{Cov}}}

\newcommand{\eqnh}{\begin{eqnarray*}}
\newcommand{\eqne}{\end{eqnarray*}}
\newcommand{\eqnhn}{\begin{eqnarray}}
\newcommand{\eqnen}{\end{eqnarray}}
\newcommand{\equh}{\begin{equation}}
\newcommand{\eque}{\end{equation}}

\def\summ#1#2#3{\sum_{#1 = #2}^{#3}}
\def\prodd#1#2#3{\prod_{#1 = #2}^{#3}}
\def\sif#1#2{\sum_{#1=#2}^\infty}

\newcommand{\eqd}{\stackrel{d}{=}}

\def\topp#1{^{(#1)}}

\def\abs#1{\left|#1\right|}

\def\pp#1{\left(#1\right)}

\def\mmid{\;\middle\vert\;}

\def\floor#1{\left\lfloor #1 \right\rfloor}

\def\qmand{\quad\mbox{ and }\quad}

\def\qmwith{\quad\mbox{ with }\quad}
\def\mfa{\mbox{ for all }}

\def\wt#1{\widetilde{#1}}
\def\wb#1{\overline{#1}}
\def\what#1{\widehat{#1}}

\def\limn{\lim_{n\to\infty}}

\def\weakto{\Rightarrow}

\def\Z{{\mathbb Z}}

\def\R{{\mathbb R}}

\def\N{{\mathbb N}} 

\def\D{{\mathbb D}}

\newcommand{\calP}{{\mathcal P}}

\newcommand{\calN}{\mathcal{N}}
\newcommand{\calM}{\mathcal{M}}
\newcommand{\calZ}{\mathcal{Z}}

\renewcommand{\d}{{\rm d}}
\newcommand{\aswto}{\stackrel{a.s.w.}\to}

\date{}

\def\R{\mathbb{R}}
\def\C{\mathbb{C}}
\def\N{\mathbb{N}}
\def\P{\mathbb{P}}

\def\1{\mathds{1}}
\renewcommand{\i}{{\rm i}}
\newcommand{\rmRe}{{\rm Re}}
\newcommand{\rmIm}{{\rm Im}}

\title[Permutations induced Chinese restaurant processes]{Limit theorems for random permutations  induced by Chinese restaurant processes}

\author{Jaime Garza}
\address
{
Department of Mathematical Sciences\\
University of Cincinnati\\
2815 Commons Way\\
Cincinnati, OH, 45221-0025, USA.
}
\email{garzaje@mail.uc.edu}

\author{Yizao Wang}
\address
{
Department of Mathematical Sciences\\
University of Cincinnati\\
2815 Commons Way\\
Cincinnati, OH, 45221-0025, USA.
}
\email{yizao.wang@uc.edu}

\begin{document}\sloppy
\begin{abstract}We investigate the random permutation matrices induced by the Chinese restaurant processes with $(\alpha,\theta)$-seating. When $\alpha=0,\theta>0$, the permutations are those following Ewens measures on symmetric groups, and have been extensively studied in the literature. Here, we consider $\alpha\in(0,1)$ and $\theta>-\alpha$. In an accompanying paper,  a functional central limit theorem is established for partial sum of weighted cycle counts in the form of $\summ j1n a_jC_{n,j}$, where $C_{n,j}$ is the number of $j$-cycles of the permutation matrix of size $n$. Two applications are presented. One is on linear statistics of the spectrum,  and the other is on the characteristic polynomials outside the unit circle.
\end{abstract}

\maketitle

\section{Introduction}
Many developments have been seen on the distributions of eigenvalues of random permutation matrices. It was pointed out by \citet{diaconis94eigenvalues} that in order to study these eigenvalues, it suffices to study the cycle structures of the random permutations, and for the latter many results and tools developed earlier can be of great help. Following this remark, \citet{wieand00eigenvalue} later investigated the asymptotic behavior of counting functions of eigenvalues of a uniform random permutations, and \citet{benarous15fluctuations} characterized limit fluctuations of general linear statistics of the spectrum of random permutation matrices following Ewens measures on symmetric groups with parameter $\theta>0$ (with $\theta=1$ these are the uniform random permutation matrices). The characteristic polynomials have also been investigated by \citet{hambly00characteristic}. Limit theorems for other statistics of random permutations following Ewens measures also exist in the literature, and we mention \citep{bahier19number,bahier22smooth,najnudel13distribution}.

We are interested in the asymptotic behavior of the random permutations induced by the Chinese restaurant processes \citep{pitman06combinatorial}. The Chinese restaurant processes have two parameters $(\alpha,\theta)$ (sometimes referred to as {\em with $(\alpha,\theta)$-seating}), and the induced sequence of permutations are referred to as $(\alpha,\theta)$-permutations throughout. In particular, it is well-known that an $(0,\theta)$-permutation follows Ewens measure with $\theta>0$. Here, our focus is in the regime 
\equh\label{eq:alpha}
\alpha\in(0,1) \qmand \theta>-\alpha.
\eque
From now on, by $(\alpha,\theta)$-seating or $(\alpha,\theta)$-permutation we implicitly assume \eqref{eq:alpha}; otherwise we will write $(0,\theta)$ instead.

A key difference between the case $\alpha=0$ and $\alpha\in(0,1)$ is the asymptotic behaviors of the cycle counts. For $(0,\theta)$-permutations the cycle counts converge in distribution {\em  without any normalization} to independent Poisson random variables, while for $(\alpha,\theta)$-permutations the cycle counts after appropriate normalization converge in distribution to dependent Gaussian random variables. As a consequence, the famous and powerful {\em Feller's coupling} for $(0,\theta)$-permutations (see e.g.~\citep{arratia03logarithmic}) no longer applies to $(\alpha,\theta)$-ones. 

To study $(\alpha,\theta)$-permutations, we shall exploit the well-known fact that Chinese restaurant processes are essentially equivalent to infinite urn models (with random frequencies) thanks to Kingman's representation theorem: in particular
 cycle counts (how many say $j$-cycles) of permutations can be identified with occupancy counts (how many urns have $j$ balls) of the Karlin model.
Then with $\alpha\in(0,1)$, the Chinese restaurant processes are essentially the Karlin model \citep{karlin67central,gnedin07notes} (where the frequencies in the infinite urn models decay at a regularly-varying rate with index $-1/\alpha$), and limit theorems for occupancy counts of Karlin model have been extensively investigated. 
In an accompanying paper \citep{garza24functional}, as an intermediate step we established a functional central limit theorem for general weighted occupancy processes, and we explain in Theorem \ref{thm:1} how this result can be applied to summations of weighted cycle counts of  $(\alpha,\theta)$-permutation matrices with $\alpha\in(0,1)$.  

The main contributions of this paper are, by applying the the functional central limit theorem established in \citep{garza24functional}, two limit theorems for $(\alpha,\theta)$-permutation matrices. The first is on linear statistics of eigenvalues. The second is on characteristic polynomials. In both examples, the statistics of interest can be represented as the sum of weighted cycle counts, and hence Theorem \ref{thm:1} applies directly. Our results are to be compared with counterpart ones for $(0,\theta)$-permutation matrices, which have been extensively studied in the literature. We mention representative results by  \citet{hambly00characteristic} and \citet{benarous15fluctuations}, and also a recent interesting development by \citet{coste24characteristic} (their main focus is on summation of $d$ i.i.d.~uniform permutation matrices but some discussions on $d=1$ have been very useful to us). More detailed reviews and comparisons are provided below.

A highlight of our results is that they are functional central limit theorems and in particular  they characterize the macroscopic temporal evolution of the eigenvalues, as the size of the matrix tends to infinity. A remarkable feature  of the $(0,\theta)$ and $(\alpha,\theta)$-random permutation matrices induced by the Chinese restaurant processes is that these matrices are intrinsically constructed as a sequence of temporally evolving matrices (with changing sizes) as the number of customers in the restaurant goes to infinity. Not many functional limit theorems are known for random matrices as it is a priori not clear how matrices of different dimensions are modeled jointly. In particular, Earlier results \citep{wieand00eigenvalue,benarous15fluctuations} concern only the asymptotic behavior of a sequence of permutation matrices not coupled in any sense. 
\begin{remark}It would be interesting to see if counterparts of our results hold for $(0,\theta)$-permutation matrices. Note that the Feller's coupling is not powerful enough to establish the {\em functional} central limit theorem of our type, since it represents the cycle structure at step $n$ by a simple sequence of Bernoulli random variables and loses track of the evolution of cycle structure as $n$ grows (see Remark \ref{rem:Feller no FCLT}; but the method allows to establish another type of functional central limit theorems as shown in \citet{hansen90functional}). 
 A more refined point-process convergence has been recently established by \citet{galganov24short} (actually their limit theorems apply to a much larger class of permutations), exploiting directly the definition of the permutation.  
At the same time, Kingman's representation theorem remains valid to translate permutations induced by $(0,\theta)$-seating into an infinite urn model (Karlin model with $\alpha=0$). Along this line we mention the recent paper by \citet{iksanov22small} (see also \citep[Proposition 4.1]{iksanov22functional}), where they investigated Karlin model with $\alpha=0$, and obtained new Gaussian processes as the scaling limits for various statistics including occupancy counts. They imposed a so-called $\Pi$-variation assumption on the decay of the frequencies that is not satisfied by the asymptotic frequencies induced by the $(0,\theta$)-seating. \end{remark}

One of the main motivations to study random permutation matrices came from the study of random unitary matrices (e.g.~\citep{diaconis94eigenvalues,hughes01characteristic}). 
While random unitary matrices are arguably much more important in random matrix theory, random permutation matrices have similar behaviors in certain aspects and yet have relatively simpler structure to deal with. For example, some similarities have been reported already when comparing fluctuations of characteristic polynomials in \citep{hambly00characteristic}. It is a natural question in general to examine and compare limit fluctuations of different random matrices. On the other hand, from the observation that the cycle counts of $(0,\theta)$ and $(\alpha,\theta)$-permutations have different types of asymptotic behaviors, one expects already the limit fluctuations of statistics of interest to be qualitatively different. Therefore, $(\alpha,\theta)$-permutation matrices are better viewed as examples `not too close' to $(0,\theta)$-ones. 

There are various extensions of $(0,\theta)$-permutations investigated in the literature. For models of a single matrix, we mention \citep{hughes13random,ercolani14cycle,najnudel13distribution}. Another extensively studied model is the summation of $d$ i.i.d.~copies of uniform permutation matrices, and in particular they correspond to random $d$-regular graphs; see \citep{dumitriu13functional,basak18circular}. For this model, some challenges in the extension to sum of i.i.d.~$(0,\theta)$-permutations are explained in \citep{coste24characteristic}. Investigating the summation of i.i.d.~$(\alpha,\theta)$-permutation matrices is another interesting question. Notice that as long as $d\ge 2$, the technique developed for a single matrix (Feller's coupling for $(0,\theta)$-permutations or the methodology relating to Karlin model for $(\alpha,\theta)$-permutations) no longer works, and the analysis becomes immediately much more involved. 


The paper is organized as follows. In Section \ref{sec:CRP} we review the Chinese restaurant processes and its induced permutations. In Section \ref{sec:FCLT} we prove a functional central limit theorem for weighted sums of cycle counts, which is a consequence of Kingman's representation theorem and the main result in \citep{garza24functional}. Section \ref{sec:linear} is devoted to limit theorems for linear statistics of eigenvalues. Section \ref{sec:log} is devoted to limit theorems for characteristic polynomials. 

\subsection*{Acknowledgements}
J.G.~and Y.W.~would like to thank Alexander Iksanov and Joseph Najnudel for several stimulating discussions. J.G.~and Y.W.~were partially supported by Army Research Office, US (W911NF-20-1-0139).
 Y.W.~was also partially supported by Simons Foundation (MP-TSM-00002359), and a Taft Center Fellowship (2024--2025) from Taft Research Center at University of Cincinnati.

\section{Chinese restaurant processes}\label{sec:CRP}

A Chinese restaurant process is a sequence of random partitions $(\Pi_n)_{n\in\N}$, each $\Pi_n$ a partition of $[n] = \{1,\dots,n\}$, determined by two parameters $(\alpha,\theta)$. The range $\alpha\in[0,1)$ and $\theta>-\alpha$ contains the most interesting cases.  The sequence of random partitions are constructed recursively as follows. Set $\Pi_1=\{\{1\}\}$ first. Suppose the construction of the first $n$ partitions $(\Pi_i)_{i=1,\dots,n}$ have been completed, and $\Pi_n$ having $k$ blocks of lengths $n_1,n_2,\dots,n_k>0$. This is often referred to as having $k$ tables each with $n_j$ customers, $j=1,\dots,k$, and $n_1+\cdots+n_k = n$. Then, set $\Pi_{n+1}$ by the following modification of $\Pi_n$:
        \begin{enumerate}[(i)]
        \item  placing $n+1$ in the $i$-th block of $\Pi_n$ with probability $(n_i-\alpha)/(n+\theta)$,
        \item  or creating a new block with $n+1$ with probability $(\theta +\alpha k)/(n+\theta)$,
    \end{enumerate}
    and continue the procedure iteratively.
The two actions are sometimes referred to as seating the $(n+1)$-th customer at $i$-th table or at a new table, respectively. 
This procedure yields a sequence of random partitions. In order to obtain a sequence of random permutations, one simply requires further that when a new customer sits at an occupied {\em round} table, he/she chooses a seat uniformly (say to the left of a randomly chosen customer), and the seating of all occupied round tables yield a cycle-representation of a permutation. (Having round tables as in many Chinese restaurants is convenient in this description, in order to relate to cycle structure of permutations.)
See \citep[Chapter 3]{pitman06combinatorial} for an overview of Chinese restaurant processes. We see very soon below that for the question of our interest, the exact seating is irrelevant, and it suffices to focus on the number of customers of each table.

The random partitions induced by a Chinese restaurant process have in particular two properties. First, each one is exchangeable, in the sense that for any permutation $\pi$ on $[n]$ 
 \begin{equation*}
     \P(\Pi_n=\{A_1,...,A_k\})=\P(\Pi_n=\{\pi(A_1),...,\pi(A_k)\}),
 \end{equation*}
 where $\pi(A):=\{\pi(j):j\in A\}$ and $A_1,\dots,A_k$ above is any disjoint union of $[n]$. Second, the family of random partitions are consistent, in the sense that for any $m,n\in\N, m<n$, the random partition $\Pi_m$ is exactly $\Pi_n$ restricted to $[m]$.  (That is, if $\Pi_n=\{A_1,...,A_k\}$, then  $\{A_1\cap [m],...,A_k\cap [m]\} = \Pi_m$ for all $m<n$.) Then, with these two properties the sequence $(\Pi_n)_{n\in\N}$ induce an exchangeable random partition on $\N$.

  Now let $(A_{n,i})_{i\in \N}$ denote the sequence of block sizes of $\Pi_n$ arranged in the order of appearance in the construction above, and set  $A_{n,i}=0$ if $\Pi_n$ has strictly fewer than $i$ blocks. Kingman's representation theorem for exchangeable random partitions of $\N$ \citep{kingman82coalescent}  ensures that the frequencies
 \begin{equation}\label{CRPFreq}
     P_i:=\lim_{n\rightarrow \infty} \frac{A_{n,i}}{n}, \hspace{.2in} i\in \N,
 \end{equation}
 exist almost surely, they are non-deterministic, strictly positive, and satisfy $\sif i1 P_i= 1$.  Moreover, one has the well-known stick-breaking representation of the joint law (recall that the model depends on fixed parameters $\alpha,\theta$):
 \[
 (P_1,P_2,\dots)\eqd \pp{B_1,\wb B_1B_2,\wb B_1\wb B_2B_3,\dots},
 \]
where $(B_i)_{i\in\N}$ are independent random variables with $B_i$ distributed as a beta random variable with parameter $(1-\alpha,\theta+i\alpha)$, and $\wb B_i := 1-B_i$. 
With $\alpha=0,\theta>0$, the induced random partitions are governed by Ewens' sampling formula, and the random permutations by Ewens measures on symmetric groups.

In this paper, we focus on $\alpha\in(0,1),\theta>-\alpha$. 
Let $(P_j^\downarrow)_{j\in\N}$ denote the non-increasing ordering of $(P_j)_{j\in\N}$. It is well-known that in this range,
  \begin{equation}\label{alphadivlim}
P_n^\downarrow\sim \pp{\frac{\mathsf s_{\alpha,\theta}}{\Gamma(1-\alpha)}}^{1/\alpha} n^{-1/\alpha}     \mbox{ almost surely,}
 \end{equation}
as $n\to\infty$, where $\mathsf s_{\alpha,\theta}$ is a random variable that can also be defined as the almost sure limit of $|\Pi_n|/n^\alpha$ as $n\to\infty$, with $|\Pi_n|$ being the total number of cycles of the partition $\Pi_n$. 
 (This can be found at for example \citep[Eq.~(3.40)]{pitman06combinatorial}, while Eq.~(3.35) and the formula of $Z$ right after from the same reference have a typo.)
The random variable  $\mathsf s_{\alpha,\theta}$ is strictly positive and is referred to as the {\em $\alpha$-diversity}, and it shall show up in our limit theorem later. In particular we observe that $\mathsf s_{\alpha,\theta}$ is measurable with respect to 
\equh\label{eq:calP}
\calP := \sigma\pp{(P_i^\downarrow)_{i\in\N}}.
\eque
 The distribution of $\mathsf s_{\alpha,\theta}$ follows a generalized Mittag--Leffler distribution with parameters $(\alpha,\theta)$. See \citep[Section 3.3]{pitman06combinatorial} and \cite[Section 3.2]{banderier24phase} for more details. 

\section{A functional central limit theorem for weighted sum of cycle counts}\label{sec:FCLT}
Let $\sigma$ be a generic permutation of $[n]$, we set the corresponding $n\times n$ matrix $M_n$ by $[M_n]_{i,j}= \inddd{\sigma(j) = i}$. We are interested in the eigenvalues of $M_n$, which can be related  to the cycle structure as follows. Let  $C_{n,j}$ denote the number of cycles of length $j$ (referred to as $j$-cycles below) of the permutation.  More specifically, if $C_{n,j}>0$, then each sub-matrix of $M_n$ corresponding to a $j$-cycle has eigenvalues 
\begin{equation*}
    1,e^{\i 2\pi /j},e^{\i 4\pi  /j},...,e^{\i2\pi  (j-1)/j}.
\end{equation*}
So each of these eigenvalues has multiplicity at least $C_{n,j}$, as it may also be the eigenvalue of another sub-matrix corresponding to a cycle with a different length. More precisely, each eigenvalue necessarily takes the value of $e^{\i 2\pi p/n}$  for some $p = 0,1,\dots,n-1$,  with multiplicity exactly (might be zero)
\[
\sum_{j\in\N:pj/n\in\N_0}C_{n,j}.
\]
It turned out that many statistics of interest concerning eigenvalues of a permutation matrix with cycle counts $(C_{n,1},\dots,C_{n,n})$ can be expressed as 
\equh\label{eq:weighted C}
S_{\vec a,n} :=\summ j1n a_jC_{n,j},
\eque
for some prescribed sequence of numbers $(a_j)_{j\in\N}$. For convenience, we assume $a_0 = 0$ throughout.

Now let us consider permutation matrices induced by Chinese restaurant processes. It is well-known that  for a $(0,\theta)$-permutation matrix of size $n$, we have
\equh\label{eq:Feller}
(C_{n,1},\dots,C_{n,n},0,\dots) \weakto \pp{W_1,W_2,\dots}
\eque
where on the right-hand side $(W_j)_{j\in\N}$ are independent Poisson random variables each with parameter $\theta/j$. Moreover, Feller's coupling provides a way to construct $(C_{n,1},\dots,C_{n,n})$ and $(W_1,\dots,W_n)$ on the same probability space and with an explicit rate on $\esp |C_{n,j}-W_j|\to 0$, which is very convenient when proving limit theorems, as then most of the limit theorems becomes ones for independent Poisson random variables.

For $(\alpha,\theta)$-permutations, instead of \eqref{eq:Feller} we have
\equh\label{eq:C conv}
\frac1{n^{\alpha/2}}\pp{C_{n,j} - \esp\pp{C_{n,j}\mid\calP}}_{j\in\N}\aswto \pp{\frac{\mathsf s_{\alpha,\theta}}{\Gamma(1-\alpha)}}^{1/2} \pp{\calZ_{\alpha,j}}_{j\in\N},
\eque
where the convergence is the almost sure weak convergence with respect to $\calP$ the $\sigma$-algebra of asymptotic frequencies (see \eqref{eq:calP}), and the limit Gaussian process $(\calZ_{\alpha,j})_{j\in\N}$ is defined as follows. 
Let $\calM_\alpha$ be a Gaussian random measure defined on $\R_+\times \Omega'$ with intensity measure $\alpha r^{-\alpha-1}\d r\d \P'$ and $(N'(t))_{t\geq0}$ be a standard Poisson process on $(\Omega',\P')$. 
     Set
\equh\label{eq:calZ_j}
\calZ_{\alpha,j}(t)\equiv \calZ_{\alpha,\vec a,j}(t):=\int_{\R_+\times\Omega'} \pp{\inddd{N'(rt)=j}-\proba'(N'(rt) = j)}\calM_\alpha(\d r,\d\omega'), t\ge 0.
\eque
We refer to \citep{samorodnitsky94stable} for background on stochastic integrals with respect to Gaussian  random measures. 
We just recall that if $\calM$ is a Gaussian random measure on $(E,\mathcal E)$ with control measure $\mu$, then $\int_E f\d \calM$ is a centered Gaussian random variable for all $f\in L^2(E,\mu)$ and $\cov(\int_E f\d\calM, \int_E g\d\calM) = \int_E fg\d\mu$ for all $f,g\in L^2(E,\mu)$. 
Moreover, assume throughout that the Gaussian random measure $\calM_\alpha$ is independent from $\calP$, and hence $\calZ_{\alpha,j}$ is independent from $\mathsf s_{\alpha,\theta}$. 
Central limit theorems for $(\alpha,\theta)$-permutations are commonly stated in the notion of almost sure weak convergence. 
Given a Polish space $(\mathcal X,d)$ and random elements $(X_n)_{n\in\N}, X$ in it, we say $X_n$ converges almost surely weakly to $X$ as $n\to\infty$ with respect to $\calP$, denoted by $X_n\aswto X$ with respect to $\calP$, if for all continuous and bounded functions $g$ on $\mathcal X$,
\[
\limn \esp \pp{g(X_n)\mid \calP} = \esp \pp{g(X)\mid\calP} \mbox{ almost surely.}
\]
For a slightly more general notion, see \citet{grubel16functional}.

From \eqref{eq:C conv}, the limit of weighted sum of cycle counts is easy to guess: it takes the form (up to a multiplicative random constant)
\equh\label{eq:calZ}
\calZ_{\alpha}(t)\equiv \calZ_{\alpha,\vec a}(t):=\int_{\R_+\times\Omega'} \pp{a_{N'(rt)}-\esp'a_{N'(rt)}}\calM_\alpha(\d r,\d\omega'), t\ge 0,
\eque
depending on $\vec a= (a_j)_{j\in\N}$ with $a_0 = 0$. Note that the integrand can also be written as
\[
a_{N'(rt)} - \esp'a_{N'(rt)} = \sif \ell1 a_\ell \pp{\inddd{N'(rt) = \ell} - \proba'(N'(rt) =\ell)}.
\]
It was shown in \citep{garza24functional} that if $
|a_j|\le Cj^\beta$ with $\beta\in[0,\alpha/2)$, this is a family of centered Gaussian random variables, and if furthermore 
\[
|a_{i+j}-a_i|\le Cj^\beta \mfa i,j\in\N,
\]
with $\beta\in[0,\alpha/2)$,
 then the process has a continuous version that is $\gamma$-H\"older-continuous for all $\gamma\in(0,\alpha/2)$. For all our applications later, the sequence $(a_j)_{j\in\N}$ turned out to be bounded (i.e.~$\beta=0$). Explicit formula for covariance function of $\calZ_\alpha$ can be found in \citep[Section 2]{garza24functional}, although the formula is complicated most of the time.

The convergence \eqref{eq:C conv} can be already read from the literature: it follows from \citet{karlin67central} and Kingman's representation theorem (for the latter we shall see another application in the proof of Theorem \ref{thm:1} below). However, from there to prove a limit theorem for \eqref{eq:weighted C}, which is provided in Theorem \ref{thm:1} below, turned out to be technically quite challenging, and the difficulty came essentially from proving tightness of the infinite series, as fully explained in \citep{garza24functional}.

\begin{theorem}\label{thm:1}
Consider \eqref{eq:weighted C} based on $(C_{n,j})_{n\in\N, j=1,\dots,n}$ from $(\alpha,\theta)$-permutation matrices. Assume that $(a_j)_{j\in\N}$ is a sequence of real numbers such that $|a_{i+j}-a_i|\le Cj^\beta$ for all $i,j\in\N$ for some $\beta\in[0,\alpha^2/2)$. Then, 
\equh\label{eq:1}
\pp{\frac{S_{\vec a,\floor {nt}} - \esp \pp{S_{\vec a,\floor{nt}}\mid\calP}}{n^{\alpha/2}}}_{t\in[0,1]}\aswto \pp{\frac{\mathsf s_{\alpha,\theta}}{\Gamma(1-\alpha)}}^{1/2}\pp{\calZ_{\alpha}(t)}_{t\in[0,1]}
\eque
with respect to $\calP$ in $D([0,1])$, where on the right-hand side the process $\calZ_{\alpha}$ is as in \eqref{eq:calZ}   and assumed to be independent from $\calP$. 
\end{theorem}
\begin{proof}
The joint law of 
\[
(C_{n,1},\dots,C_{n,n})
\]
is related to occupancy statistics from an infinite urn scheme via Kingman's representation theorem. We first recall the setup of an infinite urn model with sampling frequencies $(p_j)_{j\in\N}$ (so $p_j>0$ and $\sif j1 p_j = 1$). Consider i.i.d.~random variables $(Y_i)_{i\in\N}$ with  $\proba(Y_i = j) = p_j,j\in\N$ representing the label of urn the ball is thrown into at each round, and by time $n$ the following statistics:
\equh\label{eq:D}
K_{n,\ell} := \summ i1n \inddd{Y_i = \ell}, \quad \ell\in\N\qmand
D_{n,j}:=\sif \ell1 \inddd{K_{n,\ell} = j}, \quad j\in\N. 
\eque
Recall also the asymptotic frequencies $(P_i)_{i\in\N}$ in \eqref{CRPFreq} and $\calP$ in \eqref{eq:calP}, and consider the infinite urn scheme with random sampling frequencies $(P_i)_{i\in\N}$. That is, the random variables $(Y_i)_{i\in\N}$ now are conditionally independent and identically distributed given $\calP$, following
\equh\label{eq:Y conditional}
\proba(Y_i = j\mid\calP) = P_j, \mbox{ almost surely, }\quad j\in\N.
\eque
 Then, Kingman's representation theorem says that the law of $(C_{n,1},\dots,C_{n,n})_{n\in\N}$ given $\calP$ is the same as the law of $(D_{n,1},\dots,D_{n,n})_{n\in\N}$ from an infinite urn scheme with sampling frequencies $(P_i)_{i\in\N}$ (i.e., given $\calP$ we first sample $(Y_i)_{i\in\N}$ via \eqref{eq:Y conditional} and then consider $D_{n,j}$ as in \eqref{eq:D}). In particular, for continuous and bounded functions $f:\R^n\to\R$, 
\[
\esp \pp{f(C_{n,1},\dots,C_{n,n})\mmid\calP} = \esp \pp{f(D_{n,1},\dots,D_{n,n})\mmid\calP},
\]
almost surely. Then, the limit theorem concerning $\summ j1n a_jC_{n,j}$ becomes a limit theorem concerning 
\[
\summ j1n a_jD_{n,j},
\]
under the conditional law of $(D_{n,1},\dots,D_{n,n})$ with respect to $\calP$. 

Notice that \citep[Theorem 1.1]{garza24functional} is a limit theorem for $\summ j1n a_j (D_{n,j}- \esp(D_{n,j}\mid\calP))$ (for almost sure weak convergence; in fact the main result therein is for Karlin model with fixed sampling frequencies, but the extension to one with random frequencies is immediate), which now translates into a limit theorem for $\summ j1n a_j(C_{n,j}-\esp (C_{n,j}\mid\calP))$. It remains to check that the decaying rate of $(P_j^\downarrow)_{j\in\N}$ (that is, the frequencies $(P_j^\downarrow)_{j\in\N}$ ordered in non-increasing order) satisfies the assumption therein. In particular, we need
\[
P_j^\downarrow\sim \mathsf C_0j^{-1/\alpha},
\]
as $j\to\infty$ for some constant $\mathsf C_0>0$. It is clear that here, the constant $\mathsf C_0$ is allowed to be a random variable measurable with respect to $\calP$ to apply the theorem, and it is indeed the case as reviewed in  \eqref{alphadivlim}.
Therefore \eqref{eq:1} follows from \citep[Theorem 1.1]{garza24functional} (note also that in the notations therein, $\sigma_n \sim (\mathsf C_0^\alpha n^\alpha)^{1/2}$). 
\end{proof}

\subsection{Comparison with \texorpdfstring{$(0,\theta)$}{theta}-permutations}
\citet{benarous15fluctuations} when studying the linear statistics of eigenvalues also established limit theorems for weighted sum of cycle counts $\summ j1n a_jC_{n,j}$ for $(0,\theta)$-permutations. We provide a comparison.
\begin{remark}\label{rem:Feller no FCLT}
We first explain briefly why Feller's coupling does not lead to a functional central limit theorem of the type of Theorem \ref{thm:1} immediately. Recall that the Feller's coupling for $(0,\theta)$-permutations (more precisely the coupling is on the cycle structure) says the following. With a sequence of independent $(\xi_i)_{i\in\N}$ each distributed as a Bernoulli random variable with parameter $\theta/(\theta+i-1)$, (i) the random variables
\[
\what C_{n,j}:=\summ i1{n-j}\xi_i(1-\xi_{i+1})\cdots(1-\xi_{i+j-1})\xi_{i+j} + \xi_{n-j+1}(1-\xi_{n-j+2})\cdots(1-\xi_n), j=1,\dots,n,
\]
have the same joint law as $C_{n,1},\dots,C_{n,n}$, (ii) the random variables $W_j:=\sif i1\xi_i(1-\xi_{i+1})\cdots(1-\xi_{i+j-1})\xi_{i+j}, j\in\N$ are independent Poisson random variables each with parameter $\theta/j$, and (iii) the rate $\esp |C_{n,j}-W_j|\to 0$ can be computed explicitly. Exploiting Feller's coupling, formally we have
\equh\label{eq:Poisson regime}
 \summ j1n a_j(C_{n,j} - \esp C_{n,j}) \approx \summ j1n a_j (W_j - \esp W_j)\weakto \sif j1 a_j\pp{W_j - \frac\theta j}.
\eque
For the right-hand side to converge, the Kolmogorov's three series theorem says the necessary and sufficient condition is $\sif j1 a_j^2/j<\infty$, and this condition turns out to be sufficient to make the step `$\approx$' rigorous. (One can also show the corresponding convergence in \eqref{eq:Poisson regime} without centering by essentially the same argument.)

However, Feller's coupling cannot apply to establish the functional central limit theorem of our interest: it suffices to compare $\what C_{\floor{ns},j}$ and $\what C_{\floor{nt},j}$ with $s<t$. By construction, we have $\what C_{\floor{nt},j}\ge \what C_{\floor{ns},j}-1$ almost surely. But clearly for $C_{\floor{ns},j}$ and  $C_{\floor{nt},j}$ the corresponding relation does not hold: there is a non-zero probability for the event $C_{\floor{nt},j} = 0$, regardless the value of $C_{\floor{ns},j}$. 
\end{remark}
A phase transition for limit theorems of $\summ j1n a_jC_{n,j}$ was revealed in \citep{benarous15fluctuations}. The authors characterized two phases completely therein (their Theorems 2.3 and 2.4 are the first two cases below), while in fact there is also a third phase. 
(Their $R_j(f)$ is the same as ours below in \eqref{eq:a_j(f)}, while our $a_j \equiv a_j(f) = jR_j(f)$ corresponds to their $u_j$ used in the proofs.) 
\begin{enumerate}[(i)]
\item When $\sif j1 a_j^2/j<\infty$, and when $\theta\in(0,1)$ under the additional assumption that $|a_j|$ converges to zero in the Cesaro $(C,\theta)$ sense (see their Section 2 for more details), by the argument around \eqref{eq:Poisson regime} we have
\[
\sif j1 a_j (C_{n,j} - \esp C_{n,j}) \weakto \sif j1 a_j\pp{W_j - \frac\theta j}.
\]
\item When $\sif j1 a_j^2/j = \infty$ and 
\equh\label{eq:2.10}
\max_{j=1,\dots,n}|a_j| = o\pp{\pp{\summ j1n a_j^2/j}^{1/2}},
\eque
we have
\[
\frac1{\sigma_n}\summ j1n a_j (C_{n,j} - \esp C_{n,j})\weakto \calN(0,1),
\]
where $\sigma_n = (\summ j1n a_j^2/j)^{1/2}$ and the right-hand side above denotes the standard normal distribution. 
\item Once \eqref{eq:2.10} is violated, and the limit is no longer Gaussian; see \citep{garza24solo}. Note that \eqref{eq:2.10} is in fact very restrictive: if $a_j\sim Cj^\delta$ for any $\delta>0$, one readily checks that \eqref{eq:2.10} is violated. 
\end{enumerate}

We now see that for $(\alpha,\theta)$-permutations the phase transition is drastically different from the $(0,\theta)$-ones. First, there is no regime of the weighted sum of Poisson random variables as in their model. That regime is essentially a consequence of the Poisson convergence of cycle counts \eqref{eq:Feller}, as explained around \eqref{eq:Poisson regime}. It is clear that for $(\alpha,\theta)$-permutations this regime does not exist.

Our main theorem formally corresponds to their second phase above, the Gaussian one, and yet the conditions are quite different.  Note that our assumption $|a_j|\le Cj^\beta$ with $\beta<\alpha^2/2$ (it is plausible that the condition can be relaxed to $\beta<\alpha/2$ although we have not found a proof) implies $\sif j1 a_j^2/j^{\alpha+1}<\infty$. We expect a phase transition for $(\alpha,\theta)$-permutations when $\sif j1 a_j^2/j^{\alpha+1} = \infty$, and this question is left for a future study.

\section{Limit fluctuations for linear statistics of eigenvalues}\label{sec:linear}
Let $M_n \equiv M_n\topp{\alpha,\theta}$ be an $(\alpha,\theta)$-permutation matrix. 
For each $M_n$, let 
\[
e^{\i2\pi\theta_{n,1}},\dots,e^{\i2\pi\theta_{n,n}}
\]
represent its $n$ eigenvalues, with $\theta_{n,j}\in[0,1),j=1,\dots,n$. 
Let $f:[0,1)\to \R$ be a test function. We are interested in the asymptotic behavior of the following linear statistics of the spectrum of a matrix $M_n$:
\[
S_n(f):=\summ i1n f(\theta_{n,i}).
\]
Alternatively, exploiting the relation between cycle structure and the related eigenvalues, we have
\[
S_n(f) = \summ j1n C_{n,j}\summ k0{j-1}f\pp{\frac kj}.
\]
We recognize a Riemann sum on the right-hand side.
Introduce
\equh\label{eq:a_j(f)}
a_j(f):=jR_j(f) \qmwith R_j(f) := \frac1j \summ k0{j-1}f\pp{\frac kj} - \int_0^1f(x)\d x, j\in\N.
\eque
We then have
\[
S_n(f) = \summ j1n C_{n,j} jR_j(f) + \summ j1n C_{n,j} j \int_0^1 f(x)\d x = \summ j1n a_j(f) C_{n,j}  +  n\int_0^1 f(x)\d x.
\]
That is, we can express
\[
S_n(f) - \esp (S_n(f)\mid\calP) = \summ j1n a_j(f) (C_{n,j} - \esp (C_{n,j}\mid\calP)),
\]
and we are in the exact situation to apply Theorem \ref{thm:1}.  The following is an immediate corollary.

\begin{theorem}
Suppose there exists $\beta\in[0,\alpha^2/2)$ such that $|a_{i+j}(f) - a_i(f)|\le Cj^\beta$ for all $i,j\in\N$. Then, 
\[
\pp{\frac{S_{\floor {nt}}(f) - \esp \pp{S_{\floor{nt}}(f)\mid\calP}}{n^{\alpha/2}}}_{t\in[0,1]}\aswto \pp{\frac{\mathsf s_{\alpha,\theta}}{\Gamma(1-\alpha)}}^{1/2}\pp{\calZ_{\alpha}(t)}_{t\in[0,1]}
\]
with respect to $\calP$ in $D([0,1])$, where on the right-hand side the process $\calZ_{\alpha}$ is as in \eqref{eq:calZ} with $a_j = a_j(f)$  and assumed to be independent from $\calP$. 
\end{theorem}

It remains to discuss which class of test functions $f$ are such that $(a_j(f))_{j\in\N}$ satisfies the assumption therein. 
The simplest example of $f$ is probably $f(x) = \inddd{x\in[c_1,c_2)}$ for some $0\le c_1<c_2\le 1$. In this case, $S_n(f)$ is  the counting number of eigenvalues falling in the arc $[2\pi c_1,2\pi c_2)$, and $R_j(f) = (\{jc_1\}-\{jc_2\})/j$ and hence $|a_j(f)| \le 1$. 
\citet{wieand00eigenvalue} first investigated this question for uniform random permutation matrices, and \citet{benarous15fluctuations} considered the question for general test functions $f$. 
We just mention one more class of test functions: if $f\in W^{1,p}$ (the space of absolutely continuous function $f:[0,1]\to\R$ with $f'\in L^p([0,1])$), then $a_j(f)\to 0$ \citep{cruz-uribe02sharp}. For more discussions, see \citep[Section 3, in particular Lemma 3.1]{benarous15fluctuations}. In particular, conditions with $|a_j(f)|\le Cj^\beta$ for some $\beta>0$ can be related to the regularity of test function $f$; however it is not clear to us whether a control on $|a_{i+j}(f)-a_i(f)|\le C j^\beta$ can also be related. 
\section{Limit fluctuations for characteristic polynomials}\label{sec:log}
Let $M_n \equiv M_n\topp{\alpha,\theta}$ be a random permutation matrix corresponding to the $(\alpha,\theta)$-seating, and $\lambda_1(M_n),\dots,\lambda_n(M_n)$ be its eigenvalues. 
Consider the (transformed) characteristic polynomial
\[
\what\chi_n(z):= \det(I-zM_n) = \prodd j1n \pp{1-z\lambda_j(M_n)} = \prodd j1n\pp{1-z^j}^{C_{n,j}},
\]
where in the last equality we recall that each $j$-cycle contributes a factor $(1-z^j)$ in the characteristic polynomial. 
Write the empirical spectral distribution $\mu_n :=n\inv\summ j1n\delta_{\lambda_j(M_n)}$. It is known that $\mu_n$ converges to the uniform probability measure on the unit circle of $\C$, denoted by $\partial\mathbb D_1$ where $\D_r :=\{z\in\C: |z|<r\}$ is the open disc in the complex plane of radius $r>0$ \citep{najnudel13distribution}.

It is natural to consider
\[
\log\what\chi_n(z) = \summ j1n\log\pp{1-z^{j}}C_{n,j}, z\in\D_1,
\]
where the value of $\log(1-z^j)$ is taken as the principle branch with $\log 1 = 0$. 
The most interesting question is the limit fluctuations of $\what \chi_n(z)$ with $z\in\partial\D_1$. However, our Theorem \ref{thm:1} does not include this case since $a_j(z) = \log(1-z^j)$ is unbounded for $z$ irrational (see Remark \ref{rem:circle} for more discussions). Instead, we have an application of Theorem \ref{thm:1} for $z\in\D_1$ and $z\in\C\setminus\wb\D_1$.
Write
\equh\label{eq:eta}
\eta_\alpha(z,t):=\int_{\R_+\times\Omega'} \pp{\what\log\pp{ 1-z^{N'(rt)}}-\esp'\what\log\pp{1-z^{N'(rt)}}}
\calM_\alpha(\d r,\d\omega'), 
\eque
for $t\ge0, z\in\D_1$, 
where $\what \log (1-z)  = 0$ if $z=1$ (this occurs only when $N'(rt) = 0$ and when computing $\esp'\what\log(1-z^{N'(rt)})$), and otherwise $\what \log (1-z) = \log (1-z)$.
Writing the integrand on the right-hand side of \eqref{eq:eta} as $g_{t,z}(r,\omega')$, we understand  $\eta_\alpha(z,t) = \rmRe\eta_\alpha(z,t)+\i \rmIm\eta_\alpha(z,t)$ with 
\[
\pp{\rmRe\eta_\alpha(z,t),\rmIm\eta_\alpha(z,t)} = \pp{\int \rmRe g_{z,t}\d
\calM_\alpha, \int \rmIm g_{z,t}\d
\calM_\alpha},
\]
and in particular the real and imaginary parts are not independent. For complex random variables $(Z_n)_{n\in\N}$ and $Z$, we write $Z_n\weakto Z$ if $(\rmRe Z_n,\rmIm Z_n)\weakto (\rmRe Z,\rmIm Z)$ as random vectors, and similar interpretations for convergence of complex-valued stochastic processes.

Alternatively, using the expression of $\calZ_{\alpha,j}(t)$ in \eqref{eq:calZ_j}, we have
\equh\label{eq:weighted Z_j}
\eta_\alpha(z,t) = \sif j1 \log(1-z^j)\calZ_{\alpha,j}(t).
\eque
\begin{theorem}\label{thm:ch poly}
We have
\equh\label{eq:log p}
\pp{\frac{{\log\what\chi_{\floor{nt}}(z) - \esp (\log\what\chi_{\floor{nt}}(z)\mid\calP)}}{n^{\alpha/2}}}_{z\in\D_1,t\in[0,1]}
\aswto\pp{\frac{\mathsf s_{\alpha,\theta}}{\Gamma(1-\alpha)}}^{1/2}\pp{\eta_\alpha(z,t)}_{z\in\D_1,t\in[0,1]}.
\eque
as $n\to\infty$, where the convergence is in the sense that for all any $z_1,\dots,z_d$ from the prescribed domain, we have the almost sure weak convergence with respect to $\calP$ of $d$ stochastic processes indexed by $(z_j,t), t\in[0,1]$ for $j=1,\dots,d$ in the space $D([0,1])^d$.
\end{theorem}
For $z\in\C\setminus\wb\D_1$, we can write 
\[
\what\chi_n(z) = \prodd j1n (z^j(z^{-j}-1))^{C_{n,j}} =  z^n\prodd j1n (z^{-j}-1)^{C_{n,j}}.
\] 
The term $z^n$ corresponds to a drift term when studying $\log\what\chi_n(z)$, and the analysis of the product on the right-hand side above is essentially the same as the one of $\what\chi_{n}(z\inv)$. We omit the details.
\begin{proof} Note that $\log\what\chi_{\floor{nt}}(z)$ is a complex-valued random variable, and to show its convergence we show the convergence of
\begin{multline*}
\pp{\rmRe\pp{\log\what\chi_{\floor{nt}}(z)},\rmIm\pp{\log\what\chi_{\floor{nt}}(z)}} \\
= \pp{\summ j1{\floor{nt}}\rmRe\pp{\log\pp{1-z^{j}}}C_{\floor{nt},j},\summ j1{\floor{nt}}\rmIm\pp{\log\pp{1-z^{j}}}C_{\floor{nt},j}},
\end{multline*}
after normalization, as a stochastic process indexed by $t\ge 0$ for each $z\in\D_1$ fixed.  This convergence then follows from Theorem \ref{thm:1} (note that it is crucial we need $z\in\D_1$ so that $\log|1-z^j|$ is bounded for all $j\in\N$), and the joint convergence in $z$
 follows from the Cram\'er--Wold device.
\end{proof}

We see the limit behaviors are drastically different from the corresponding ones of  $(0,\theta)$-permutations. 
First, we can rewrite $\eta_\alpha(z,t)$ as
\[
\eta_\alpha(z,t)  = \sif j1 \log(1-z^j)\calZ_{\alpha,j}(t) = -\sif j1\sif k1 \frac{z^{jk}}k j\calZ_{\alpha,j}(t) = -\sif \ell1 \frac{z^\ell}{\ell}\sum_{j\mid \ell}j\calZ_{\alpha,j}(t). 
\]
Then, \eqref{eq:log p} is to be compared with the $(0,\theta)$-seating. In this case, we know that 
\equh\label{eq:ch p Ewens}
 \log \what\chi_n(z) =   \summ j1n\log(1-z^j)C_{n,j}\weakto \sif j1 \log(1-z^j)W_j = -\sif \ell1\frac{z^\ell}\ell\sum_{j | \ell}j W_j,
\eque
for $z\in\D_1$, 
where $\{W_j\}_{j\in\N}$ are independent Poisson random variables with parameter $\theta/j$ respectively.   We could not locate the first appearance of this result. This convergence follows from the Poisson convergence of the cycle counts. A much more involved convergence for sum of $d$ i.i.d.~$(0,1)$-permutation matrices can be found in \citep{coste24characteristic}. For the limit random series, see also \citep{coste23sparse}.

Note that from \eqref{eq:ch p Ewens} one could see certain logarithmic correlation structure on the circle $\partial\D_1$, which has been known to be a common theme for several random matrix models. Indeed, one can approximate $\sif j1\log(1-z^j)W_j$ by $-\sif j1z^jW_j$ (the difference is easily seen to be a convergent random series when $|z|<1$) and see immediately that the latter is of covariance $\cov(\sif j1z^jW_j,\sif j1 w^jW_j) = \log(1-zw)$ (see for example \citep[Proposition 2.5]{coste24characteristic}).


One may compute explicitly the covariance function of $\eta_\alpha(z,t)$ using \eqref{eq:eta}, although the formula does not seem to be as explicit and informative as the one of the limit in \eqref{eq:ch p Ewens}. Note also that the formula \eqref{eq:weighted Z_j} is in parallel to \eqref{eq:ch p Ewens}, but a key difference is that now the processes $(\calZ_{\alpha,j})_{j\in\N}$ are correlated, while $(W_j)_{j\in\N}$ are independent. We include a calculation of $\cov(\eta_\alpha(z,1),\eta_\alpha(w,1))$ here.
Let $Q_\alpha$ denote an $\alpha$-Sibuya random variable. That is, an $\N$-valued random variable with 
\[
\proba(Q_\alpha = k) = \frac\alpha{\Gamma(1-\alpha)}\frac{ \Gamma(k-\alpha)}{\Gamma(k+1)}, \quad k\in\N.
\] 
\begin{lemma}For all $z,w\in \D_1$, 
\begin{multline*}
\cov(\eta_\alpha(z,1),\eta_\alpha(w,1))\\
= \Gamma(1-\alpha)\sif u1\sif v1 \frac1{uv}\pp{\esp(z^uw^v)^{Q_\alpha} -2^\alpha\pp{\esp \pp{\frac{z^u+w^v}2}^{Q_\alpha}  - \esp \pp{\frac {z^u}2}^{Q_\alpha} - \esp \pp{\frac{w^v}2}^{Q_\alpha}}}.
\end{multline*}
\end{lemma}
One may re-write the formula using $\esp z^{Q_\alpha} = 1-(1-z)^\alpha$ for all $|z|<1$, but this does not seem to further simplify the covariance function. 
\begin{proof}
By properties of stochastic integrals, 
\begin{align*}
\cov&(\eta_\alpha(z,1),\eta_\alpha(w,1))\nonumber\\
& = \int_0^\infty \pp{\esp\pp{\what\log(1-z^{N(r)})\what\log(1-w^{N(r)})}-\esp\what\log(1-z^{N(r)})\esp \what\log(1-w^{N(r)})}\alpha r^{-\alpha-1}\d r.
\end{align*}
First, 
\begin{align*}
\int_0^\infty & \esp\pp{\what\log(1-z^{N(r)})\what\log(1-w^{N(r)})}\alpha r^{-\alpha-1}\d r
\\
& = \alpha \int_0^\infty \sif k1 \log(1-z^k)\log(1-w^k)\frac{r^k}{k!}e^{-r}r^{-\alpha-1}\d r \\
& = \alpha \sif k1 \log(1-z^k)\log(1-w^k)\frac{\Gamma(k-\alpha)}{\Gamma(k+1)}= \Gamma(1-\alpha)\esp \pp{\log(1-z^{Q_\alpha})\log(1-w^{Q_\alpha})}.
\end{align*}
One may further rewrite the above by using $-\log(1-z) = \sif u1 z^u/u$ and obtain
\begin{align*}
\alpha \sif k1 \log(1-z^k)\log(1-w^k)\frac{\Gamma(k-\alpha)}{\Gamma(k+1)} & = \alpha\sif k1 \sif u1\sif v1\frac1{uv}(z^u)^k(w^v)^k\frac{\Gamma(k-\alpha)}{\Gamma(k+1)}\\
& = \Gamma(1-\alpha)\sif u1\sif v1 \frac1{uv}\esp(z^uw^v)^{Q_\alpha}. 
\end{align*}
Next,
\begin{align*}
\int_0^\infty    \esp&\what\log(1-z^{N(r)})\esp\what\log(1-w^{N(r)})\alpha r^{-\alpha-1}\d r\\
& = \alpha\int_0^\infty \sif{ k,\ell}1 \sif {u,v}1 \frac1{uv} z^{uk}w^{v\ell}\frac{r^k}{k!}\frac{r^\ell}{\ell!}e^{-2r}r^{-\alpha-1}\d r\\
& = \alpha\sif {u,v}1\frac1{uv}\sif\ell2 2^{\alpha-\ell}\summ k1{\ell-1}z^{uk}w^{v(\ell-k)}\binom \ell k \frac{\Gamma(\ell-\alpha)}{\Gamma(\ell+1)}.
\end{align*}
Recognizing the summation $\summ k1{\ell-1}$ as a binomial expansion (without $k=0,\ell$ two terms), we see that the above is the same as
\begin{multline*}
\alpha 2^\alpha\sif {u,v}1\frac1{uv}\sif\ell2 \pp{\pp{\frac{z^u+w^v}2}^\ell - \pp{\frac {z^u}2}^\ell - \pp{\frac{w^v}2}^\ell} \frac{\Gamma(\ell-\alpha)}{\Gamma(\ell+1)}\\
 = \Gamma(1-\alpha)2^\alpha \sif u1\sif v1\frac1{uv}\pp{\esp \pp{\frac{z^u+w^v}2}^{Q_\alpha}  - \esp \pp{\frac {z^u}2}^{Q_\alpha} - \esp \pp{\frac{w^v}2}^{Q_\alpha}}.
\end{multline*}
Combining the above we have proved the stated formula.
\end{proof}
Next, with $z\in\D_1$, by continuous mapping theorem \eqref{eq:log p} implies 
\equh\label{eq:chi convergence}
\wb \chi_n(z):=\frac{\prodd j1n (1-z^j)^{C_{n,j}/n^{\alpha/2}}}{\prodd j1n (1-z^j)^{\esp (C_{n,j}\mid\calP)/n^{\alpha/2}}}\aswto \exp\pp{ \pp{\frac{\mathsf s_{\alpha,\theta}}{\Gamma(1-\alpha)}}^{1/2} \eta_\alpha(z,1)},
\eque
as $n\to\infty$.
A counterpart of this convergence for $(0,1)$-permutation can be found in \citep[Theorem 2.2, with $d=1$]{coste24characteristic}, and therein the convergence is further  enhanced to convergence of functions in $H(\D_1)$, the space of analytical functions on $\D_1$ with the topology induced by the uniform convergence of all compact subsets of $\D_1$. The above convergence an also be strengthened to convergence in $H(\D_1)$  (a.s.w.~with respec to $\calP$) and for this purpose it suffices to establish the tightness as explained in  \citet[Sections 4 and 5.1]{coste24characteristic} (again they dealt with the summation of $d$ i.i.d.~$(0,1)$-permutation matrices), and they reported some technical challenges when extending their tightness control to $(0,\theta)$-ones. We leave the tightness for $(\alpha,\theta)$-permutation matrices to a future study.

\begin{remark}\label{rem:circle}
The most interesting case is to consider   $z\in\partial\D_1$. Obviously, $\log\what\chi_n$ is not well-defined for all rational $z\in\partial\D_1$, and the issue occurs in the real part (the imaginary part can be taken to be bounded within $[-\pi/2,\pi/2]$, and hence Theorem \ref{thm:1} applies). In the case of uniform random permutations, two types of results are known: (a) a central limit theorem holds  for $\log\what\chi_n(z)$ where $z= e^{\i 2\pi\gamma} \in\partial\D_1$ fixed and with $\gamma\in[0,1)$ irrational,  (b) if $z$ is in addition of so-called {\em finite type}, a bivariate central limit theorem holds for both the real and imaginary parts of $\log\what\chi_n(z)$, and the limit is an independent Gaussian vector, and (c) the convergence of $\log\what\chi_n(z)$ as a random generalized function to a logarithmic correlated Poisson analytical function (for (a), (b) see \citep{hambly00characteristic,zeindler13central} and for (c) see \citep{coste24characteristic}). Our approach does not apply to this case yet. The first issue is that the condition we impose for $\calZ_\alpha$ to be well-defined, $|a_j|\le Cj^\beta$, is not suitable for $a_j = {\rm Re}\log(1-z^j)$. On the other hand, our condition may not be necessary for \eqref{eq:calZ} to be a well-defined random variable. The way this issue was overcome in \citep{hambly00characteristic} was to show that the Cesaro sum of $a_j$ tends to zero, which means that despite the fact that $a_j$ is not bounded, statistically it is small most of the time. We leave the behavior of $\what\chi_n(z)$ on the unit circle to a future study. 
\end{remark}
\begin{remark}
Nevertheless, the investigation the fluctuations of characteristic polynomials outside the support of its limit is fruitful \citep{bordenave22convergence,coste24characteristic}. We have not been able to prove the tightness of $\wb\chi_n(z)$ in \eqref{eq:chi convergence} yet, though interesting stochastic processes appear to emerge in the process.  One way to go is to provide estimates on the so-called secular coefficients. Write $\wb C_{n,j} = (C_{n,j} - \esp(C_{n,j}\mid\calP))/n^{\alpha/2}$. Then, we can express the left-hand side as
\equh\label{eq:series}
\wb \chi_n(z) = \prodd j1n (1-z^j)^{\wb C_{n,j}} = \sif k0  \wb \Delta_{n,k}z^k,
\eque
and $(\wb\Delta_{n,k})_{k\in\N_0}$ are referred to as the secular coefficients of $\what\chi_n(z)$. 
The secular coefficients in this case are more complicated than in the case of $(0,\theta)$-permutations. Here, we have $\wb \Delta_{n,0} = 1$, and for $k\in\N$ by generalized Leibniz rule 
\equh\label{eq:Leibniz}
\wb \Delta_{n,k}  = \frac{\d^k}{\d z^k}\wb\chi_n(z)\Bigg\vert_{z=0} = \sum_{\substack{m_1,\dots,m_k\ge 0\\\summ j1n m_j = k}}\binom{k}{m_1,\dots,m_k}\prodd j1k \frac{\d^{m_j}}{\d z^{m_j}}\pp{(1-z^j)^{\wb C_{n,j}}}\Bigg\vert_{z=0}.
\eque
Notice that almost surely, $\wb C_{n,j}\notin\Z$, and if $m_j/j=:\wt m_j\in\N_0$, we have
\[
\frac{\d^{m_j}}{\d z^{m_j}}\pp{(1-z^j)^{\wb C_{n,j}}}\Bigg\vert_{z=0} = m_j!(-1)^{\wt m_j}\binom{\wb C_{n,j}}{\wt m_j},
\]
and otherwise the left-hand side above is zero. Combining the above with \eqref{eq:Leibniz}, we have
\[
\wb \Delta_{n,k}  = \frac{\d^k}{\d z^k}\wb\chi_n(z)\Bigg\vert_{z=0} = k!\sum_{\substack{\wt m_1,\dots,\wt m_k\ge 0\\\summ j1k j\wt m_j = k}}\prodd j1k (-1)^{\wt m_j}\binom{\wb C_{n,j}}{\wt m_j}  \weakto \mathsf P_k\pp{\wt\calZ_{\alpha,1},\dots,\wt\calZ_{\alpha,k}},
\]
where $\wt\calZ_{\alpha,k} := (\mathsf s_{\alpha,\theta}/\Gamma(1-\alpha))^{1/2}\calZ_{\alpha,k}$ and
\[
\mathsf P_k(x_1,\dots,x_k) = k!\sum_{\substack{\wt m_1,\dots,\wt m_k\ge 0\\\summ j1k j\wt m_j = k}}\prodd j1k (-1)^{\wt m_j}\binom{x_j}{\wt m_j}, k\in\N.
\]
Then formally we have (we need to prove the tightness of the series $\wb\chi_n(z)$ in \eqref{eq:series})
\equh\label{eq:1'}
\what\chi_n(z)\aswto 1+\sif k1 \mathsf P_k\pp{\wt \calZ_{\alpha,1},\dots,\wt\calZ_{\alpha,k}}z^k,
\eque
in $H(\D_1)$ with respect to $\calP$. One can check directly that the right-hand above is a different representation of the right-hand side of \eqref{eq:chi convergence}. 
\end{remark}
We conclude with a few further remarks.
\begin{remark}
Recall that the logarithmic potential of the empirical spectral distribution $\mu_n :=n\inv\summ j1n\delta_{\lambda_j(M_n)}$ is
\[
U_{\mu_n}(z):=-\int_\C\log|z-\lambda|\mu_n(\d\lambda), n\in\N.
\]
It is known that $U_{\mu_n}(z)$ converges to 
\[
U(z) = - \frac1{2\pi}\int_0^{2\pi} \log\abs{z-e^{\i \theta}}\d\theta = \begin{cases}
-\log|z|, & \mbox{ if } |z|\ge 1,\\
0, & \mbox{ if } |z|<1,
\end{cases}
\]
as the logarithmic potential of the uniform measure on the unit circle (the expression follows from Jensen's formula).  Here we follow the convention that there is a negative sign in the definition of $U_{\mu_n}$ (see e.g.~\citep{bordenave12around}). Similar to Theorem  \ref{thm:ch poly} we can establish a limit theorem for  $U_{\mu_n}(z)$ with $z\notin\partial\D_1$. We omit the details.
\end{remark}

\begin{remark}One can also show for all $z\in\C$,
\begin{multline}
\pp{\log\what\chi_n\pp{\frac z{n^{\alpha/2}}} - \esp\pp{\log \what\chi_n\pp{\frac z{n^{\alpha/2}}}\mmid\calP}}\\
\aswto z\pp{\frac{\mathsf s_{\alpha,\theta}}{\Gamma(1-\alpha)}}^{1/2}
\int_{\R_+\times\Omega'} \pp{\inddd{N'(r) = 1}-\proba'(N'(r) = 1)}\calM_\alpha(\d r,\d\omega'),\label{eq:zoom in}
\end{multline}
with respect to $\calP$. A corresponding functional central limit theorem can be established to and for the sake of simplicity we omit. 
 We sketch a proof of \eqref{eq:zoom in}. First, 
write the left-hand side as
\equh\label{eq:zoom in 1}
\summ j1n \log\pp{1-\pp{\frac z{n^{\alpha/2}}}^j}\pp{C_{n,j} - \esp\pp{C_{n,j}\mid\calP}}.
\eque
For the summand with $j=1$, write $-\log\pp{1-z/{n^{\alpha/2}}} = \sif k1(z/n^{\alpha/2})^k/k$.
It follows from Theorem \ref{thm:1} that for the term $k=1$,
\[
\frac{z}{n^{\alpha/2}}\pp{C_{n,1}-\esp(C_{n,1}\mid\calP)} 
\]
converges almost surely weakly with respect to $\calP$ to the limit in \eqref{eq:zoom in}, and, applying similar Taylor expansions to other summands in \eqref{eq:zoom in 1}, one can show that $j=1, k=1$ is the only contributing term in \eqref{eq:zoom in 1}. The details are omitted.

The convergence \eqref{eq:zoom in} corresponds to a non-trivial limit when zooming-in near $z = 0$ of the characteristic polynomial. For comparisons, a similar argument above applied to \eqref{eq:Feller} says that for $(0,\theta)$-permutation matrices, $b_n(\what\chi_n(z/b_n))\weakto W_1$ for $b_n\to\infty$ (possibly with a lower bound on the rate imposed).

A more interesting limit theorem concerning zooming-in near $e^{2\pi\i\gamma}$ for an irrational $\gamma\in[0,1)$ for $(0,\theta)$-permutations has been established by \citet{bahier19characteristic}, where the convergence of $\what\chi_n(e^{2\pi \i (\gamma+\delta/n)})$ was proved for $\delta\in\C$. This time, only the large-cycle counts contribute to the limit, and hence our method does not apply to this question for $(\alpha,\theta)$-permutations. This question is left for a future study.
 
\end{remark}

\begin{remark}
For all the discussions regarding characteristic polynomials, one may consider the natural generalization to multiplicative class functions  as investigated by \citet{zeindler13central}. 
\end{remark}

\bibliographystyle{apalike}
\bibliography{../../include/references,../../include/references18}
\end{document}